\documentclass[11pt]{smfart}

\usepackage[T1]{fontenc}
\usepackage[latin1]{inputenc}
\usepackage[french]{babel}

\usepackage[a4paper]{geometry}

\usepackage{amssymb}
\usepackage{amsmath}
\usepackage{amsthm}
\usepackage[all, 2cell, poly]{xypic}

\usepackage{xspace}
\usepackage{mathtools}
\usepackage{version}

\SelectTips{cm}{10}

\usepackage{enumerate}

\include{macros}

\author{Dimitri Ara}

\dedicatory{À la mémoire de Jean-Louis Loday}

\address{Institut de Mathématiques de Jussieu, Université Paris
Diderot -- Paris 7, Case 7012, Bâtiment Chevaleret, 75205 Paris Cedex 13, France}
\email{ara@math.jussieu.fr}
\urladdr{http://people.math.jussieu.fr/~ara/}

\altkeywords{strict $\infty$-groupoid, homotopy type, chain complex,
homotopy category, Eilenberg-Mac Lane space}

 
\title[Sur les types d'homotopie modélisés par les $\infty$-groupoïdes
stricts]{Sur les types d'homotopie modélisés\\ par les $\infty$-groupoïdes
stricts}
\alttitle{On homotopy types modelized by strict $\infty$-groupoids}

\begin{document}

\let\sbigskipamount\bigskipamount
\divide\bigskipamount by 3

\begin{abstract}
L'objet de ce texte est l'étude de la classe des types d'homotopie qui sont
modélisés par les \oo-groupoïdes stricts. Nous démontrons que la catégorie
homotopique des \oo-groupoïdes simplement connexes est équivalente à la
catégorie dérivée en degré homologique $d \ge 2$ des groupes abéliens. Nous
en déduisons que les types d'homotopie simplement connexes modélisés par les
\oo-groupoïdes stricts sont exactement les produits d'espaces
d'Eilenberg-Mac Lane. Nous étudions également brièvement le cas des
$3$-catégories avec des inverses faibles. Nous terminons par deux questions
autour du problème suggéré par le titre de ce texte.
\end{abstract}

\begin{altabstract}
\noindent
The purpose of this text is the study of the class of homotopy types which
are modelized by strict \oo-groupoids. We show that the homotopy category of
simply connected \oo-groupoids is equivalent to the derived category in
homological degree $d \ge 2$ of abelian groups. We deduce that the simply
connected homotopy types modelized by strict \oo-group\-oids are precisely the
products of Eilenberg-Mac Lane spaces. We also briefly study
$3$\nobreakdash-cat\-e\-go\-ries with weak inverses. We finish by two
questions about the problem suggested by the title of this text.
\end{altabstract}

\maketitle

\tableofcontents

\let\bigskipamount\sbigskipamount

\section*{Introduction}

Dans \emph{Pursuing Stacks} (\cite{GrothPS}), Grothendieck se propose de
généraliser le fait classique suivant : le foncteur groupoïde fondamental
$\Pi_1 : \Top \to \Gpd$, de la catégorie des espaces topologiques vers la
catégorie des groupoïdes, induit une équivalence de catégories
$\overline{\Pi_1} : \Hot_1 \to \Ho(\Gpd)$ entre la catégorie des $1$-types
d'homotopie et la localisation de la catégorie $\Gpd$ par les équivalences
de groupoïdes. Il conjecture l'existence d'une structure algébrique de
\oo-groupoïdes (faibles) et d'un foncteur \oo-groupoïde fondamental qui
induirait une équivalence de catégories entre la catégorie homotopique
$\Hot$ et une localisation de la catégorie des \oo-groupoïdes (faibles)
(voir \cite{AraWGpd} pour un énoncé précis de cette conjecture).

Grothendieck est conscient du fait que les \oo-groupoïdes stricts ne
permettent pas de réaliser cette équivalence de catégories. En effet, il lui
apparaît comme évident le fait que les types d'homotopie simplement connexes
modélisés par ceux-ci sont exactement les produits d'espaces d'Eilenberg-Mac
Lane : «~At first sight it had seemed to me that the Bangor group had indeed
come to work out (quite independently) one basic intuition of the program I
had envisioned in those letters to Larry Breen.  [\ldots] But finally it
appears this is not so, they have been working throughout with a notion of
$\infty$\nobreakdash-groupoid too restrictive for the purposes I had in mind
(probably because they insist I guess on strict associativity of
compositions, rather than associativity up to a (given) isomorphism (or
rather, homotopy) -- to the effect that the simply connected homotopy types
they obtain are merely products of Eilenberg-MacLane spaces.~»

Ce résultat apparaît huit ans plus tard dans \cite{BrownHigginsGpdCC}, où il
est attribué à Loday. (Le résultat est exprimé en termes de complexes
croisés mais ceux-ci sont équivalents aux \oo-groupoïdes
stricts par \cite{BrownHigginsClass}.)

Dans ce texte, qui est essentiellement un remaniement du premier chapitre de
la thèse~\cite{AraThesis} de l'auteur, nous présentons une preuve
élémentaire de ce résultat, sans doute plus proche de celle que Grothendieck
avait en tête. Voici comment s'articule cette preuve. L'argument de
Eckmann-Hilton montre qu'un \oo-groupoïde strict avec un unique objet et une
unique $1$-flèche (on dira qu'un tel \oo-groupoïde est $1$-réduit) est
canoniquement un \oo-groupoïde strict en groupes abéliens.  Par un théorème
de Bourn (\cite{Bourn}), la donnée d'un tel \oo-groupoïde est équivalente à
celle d'un complexe de groupes abéliens en degré homologique $d \ge 2$. On
en déduit facilement que la catégorie homotopique des \oo-groupoïdes stricts
$1$-réduits est canoniquement équivalente à la catégorie dérivée $D_{\ge
2}(\Ab)$ de groupes abéliens en degré homologique $d \ge 2$. En utilisant un
résultat classique sur les catégories dérivées, on obtient que tout
\oo-groupoïde strict $1$-réduit est faiblement équivalent à un produit de
\oo-groupoïdes d'Eilenberg-Mac Lane. Puisque tout \oo-groupoïde simplement
connexe est faiblement équivalent à un \oo-groupoïde strict $1$-réduit, on
obtient facilement que tout foncteur de réalisation topologique raisonnable
envoie un \oo-groupoïde strict simplement connexe sur un produit d'espaces
d'Eilenberg-Mac Lane.

En fait, nous obtenons un résultat plus précis. La catégorie homotopique des
\oo-grou\-poï\-des simplement connexes est équivalente, \emph{via} un zigzag
canonique, à la catégorie dérivée $\D_{\ge 2}(\Ab)$. Ce résultat est une
conséquence facile des considérations précédentes et de l'existence de la
structure de catégorie de modèles de Brown-Golasi\'nski sur la catégorie des
\oo-groupoïdes stricts (voir \cite{BrownGolas} et \cite{AraMetWGpd}). Nous
avons pris soin de marquer clairement les résultats qui utilisent cette
structure afin d'insister sur le fait que le résultat principal n'en dépend
pas.

Une brève section de ce texte est consacrée aux \oo-groupoïdes
quasi-stricts, c'est-à-dire aux \oo-catégories admettant des inverses
faibles. Cette notion a été introduite par \hbox{Kapranov} et Voedvosky dans
\cite{KapVoe}. Les deux auteurs étaient convaincus que les \oo-groupoïdes
quasi-stricts modélisaient les types d'homotopie. Il n'en est en fait rien,
comme l'a démontré Simpson dans \cite{Simpson} (voir également le chapitre 4
de \cite{SimpsonHTHC}). Nous exposons dans ce texte un argument alternatif.
Nous montrons qu'il résulte de considérations bien connues que tout
$3$-groupoïde quasi-strict est faiblement équivalent, \emph{via} un
\emph{pseudo-foncteur}, à un $3$-groupoïde strict, et donc que les $3$-types
d'homotopie simplement connexes modélisés par les $3$-groupoïdes stricts, ou
quasi-stricts, sont les mêmes. On retrouve ainsi le résultat de Simpson, à
savoir que le $3$-type de la sphère de dimension $2$ n'est pas modélisé par
un $3$-groupoïde quasi-strict.

Nous insistons sur le fait que ce texte ne répond pas à la question suggérée
par son titre. En effet, la détermination de la classe des types d'homotopie
modélisés par les \oo-groupoïdes stricts reste à notre connaissance ouverte.
Il en est de même de la question analogue pour les \oo-groupoïdes
quasi-stricts. En particulier, il n'est pas à notre connaissance connu si
ceux-ci modélisent strictement plus de types d'homotopie que les
\oo-groupoïdes stricts.

Notre texte est organisé de la manière suivante. Dans la première section,
nous rappelons la définition de \oo-groupoïde strict et fixons la
terminologie. Dans la seconde section, nous définissons les groupes
d'homotopie des \oo-groupoïdes stricts, ainsi que leurs équivalences
faibles. Nous donnons différentes caractérisations de ces équivalences
faibles. La troisième section est consacrée à la catégorie homotopique. On y
définit notamment les notions de type d'homotopie et de $n$-type
d'homotopie. La quatrième section est le coeur de ce texte. On y démontre,
selon la stratégie exposée ci-dessus, que la catégorie homotopique des
\oo-groupoïdes simplement connexes est canoniquement équivalente à la
catégorie dérivée $\D_{\ge 2}(\Ab)$. On définit une notion de \oo-groupoïde
d'Eilenberg-Mac Lane et on montre que tout \oo-groupoïde simplement connexe
est faiblement équivalent à un produit de \oo-groupoïdes d'Eilenberg-Mac
Lane. Dans la cinquième section, on définit une notion de foncteur de
réalisation de Simpson et on démontre que les types d'homotopie simplement
connexes modélisés par les \oo-groupoïdes stricts, \emph{via} un tel
foncteur de réalisation, sont exactement les produits d'espaces
d'Eilenberg-Mac Lane. On en déduit qu'aucune sphère de dimension $n \ge 2$
n'est modélisée par un \oo-groupoïde strict. Dans la sixième section, on
définit les \oo-groupoïdes quasi-stricts.  On explique qu'il résulte de
considérations bien connues que tout $3$-groupoïde quasi-strict est
faiblement équivalent, \emph{via} un \emph{pseudo-foncteur}, à un
$3$-groupoïde strict, et donc que les $3$-types d'homotopie simplement
connexes modélisés par les $3$-groupoïdes stricts, ou quasi-stricts, sont
les mêmes. Enfin, dans la dernière section, nous détaillons les deux
questions exposées ci-dessus.

Si $C$ est une catégorie et
\[
\xymatrix@C=1pc@R=1pc{
X_1 \ar[dr]_{f_1} & & X_2 \ar[dl]^{g_1} \ar[dr]_{f_2} & &  \cdots & & X_n
\ar[dl]^{g_{n-1}} \\
& Y_1 & & Y_2 & \cdots & Y_{n-1}
}
\]
est un diagramme dans $C$, nous noterons
\[ (X_1, f_1) \times_{Y_1} (g_1, X_2, f_2) \times_{Y_2} \dots
\times_{Y_{n-1}} (g_{n-1}, X_n) \]
sa limite projective.

\section{$\infty$-groupoïdes stricts}

\begin{paragr}
La catégorie des globes $\G$ est la catégorie engendrée
par le graphe
\[
\xymatrix{
\Dn{0} \ar@<.6ex>[r]^-{\Ths{1}} \ar@<-.6ex>[r]_-{\Tht{1}} &
\Dn{1} \ar@<.6ex>[r]^-{\Ths{2}} \ar@<-.6ex>[r]_-{\Tht{2}} &
\cdots \ar@<.6ex>[r]^-{\Ths{i-1}} \ar@<-.6ex>[r]_-{\Tht{i-1}} &
\Dn{i-1} \ar@<.6ex>[r]^-{\Ths{i}} \ar@<-.6ex>[r]_-{\Tht{i}} &
\Dn{i} \ar@<.6ex>[r]^-{\Ths{i+1}} \ar@<-.6ex>[r]_-{\Tht{i+1}} &
\dots
}
\]
soumis aux relations coglobulaires
\[\Ths{i+1}\Ths{i} = \Tht{i+1}\Ths{i}\quad\text{et}\quad\Ths{i+1}\Tht{i} =
\Tht{i+1}\Tht{i}, \qquad i \ge 1.\]

La catégorie des \ndef{\oo-graphes} ou \ndef{ensembles globulaires} est la
catégorie des préfaisceaux sur $\G$. Si $X$ est un \oo-graphe, on notera
$X_n$ l'ensemble $X(\Dn{n})$ et $\Gls{i}$ (resp. $\Glt{i}$)
l'application $X(\Ths{i})$ (resp. $X(\Tht{i})$). Ainsi, la donnée
de $X$ équivaut à celle d'un diagramme d'ensembles
\[
\xymatrix{
\cdots \ar@<.6ex>[r]^-{\Gls{n+1}} \ar@<-.6ex>[r]_-{\Glt{n+1}} &
X_{n} \ar@<.6ex>[r]^-{\Gls{n}} \ar@<-.6ex>[r]_-{\Glt{n}} &
X_{n-1} \ar@<.6ex>[r]^-{\Gls{n-1}} \ar@<-.6ex>[r]_-{\Glt{n-1}} &
\cdots \ar@<.6ex>[r]^-{\Gls{2}} \ar@<-.6ex>[r]_-{\Glt{2}} &
X_1 \ar@<.6ex>[r]^-{\Gls{1}} \ar@<-.6ex>[r]_-{\Glt{1}} &
X_0
}
\]
soumis aux relations globulaires
\[\Gls{i}\Gls{i+1} = \Gls{i}\Glt{i+1}\quad\text{et}\quad\Glt{i}\Gls{i+1} =
\Glt{i}\Glt{i+1}, \qquad i \ge 1.\]
Si $X$ est un ensemble globulaire et $i, j$ sont deux entiers tels que $i
\ge j \ge 0$, on définit des applications 
$\Gls[j]{i}, \Glt[j]{i} : X_i \to X_j$ en posant
\[
\Gls[j]{i} = \Gls{j+1}\cdots\Gls{i-1}\Gls{i}
\quad\text{et}\quad
\Glt[j]{i} = \Glt{j+1}\cdots\Glt{i-1}\Glt{i}.
\]

Si $X$ est un ensemble globulaire et $n$ est un entier positif, on appellera
$X_n$ l'ensemble des \ndef{$n$-flèches} de $X$. Si $n = 0$, on appellera
également $X_0$ l'ensemble des \ndef{objets} de $X$. Si $u$ est une $n$-flèche
pour $n \ge 1$, on appellera \ndef{source} (resp. \ndef{but}) de $u$, la
$(n-1)$-flèche $\Gls{n}(u)$ (resp.  $\Glt{n}(u)$).  Pour indiquer qu'une flèche
$u$ a pour source $x$ et pour but $y$, on écrira $u : x \to y$. On dira que
deux $n$-flèches $u$ et $v$ sont \ndef{parallèles} si, ou bien $n = 0$, ou bien
$n \ge 1$ et $u, v$ ont même source et même but.
\end{paragr}

\begin{paragr}\label{paragr:def_w}
Une \ndef{\oo-précatégorie} est un \oo-graphe $X$ muni
d'applications
\[
\begin{split}
  \comp_j^i & : (X_i, \Gls[j]{i}) \times_{X_j} (\Glt[j]{i}, X_i) \to
      X_i,\quad i > j \ge 0,
      \\
  \Glk{i} & : X_i \to X_{i + 1}, \quad i \ge 0,
\end{split}
\]
vérifiant les axiomes suivants :
\begin{enumerate}
\item pour tout couple $(v, u)$ dans $(X_i, \Gls[j]{i}) \times_{X_j}
  (\Glt[j]{i}, X_i)$ avec $i > j \ge 0$, on a
  \[
  \Gls{i}(v \comp_j^i u) = 
  \begin{cases}
    \Gls{i}(u), & j = i - 1 \\
    \Gls{i}(v) \comp_j^{i-1} \Gls{i}(u), & j < i - 1
  \end{cases} ;
  \]
\item pour tout couple $(v, u)$ dans $(X_i, \Gls[j]{i}) \times_{X_j}
  (\Glt[j]{i}, X_i)$ avec $i > j \ge 0$, on a
  \[
  \Glt{i}(v \comp_j^i u) = 
  \begin{cases}
    \Glt{i}(v), & j = i - 1 \\
    \Glt{i}(v) \comp_j^{i-1} \Glt{i}(u), & j < i - 1
  \end{cases} ;
  \]
\item pour tout $u$ dans $X_i$ avec $i \ge 0$, on a
  \[
  \Gls{i+1}\Glk{i}(u) = u = \Glt{i+1}\Glk{i}(u), \qquad i \ge 0.
  \]
\end{enumerate}
Pour $i$ et $j$ deux entiers tels que $i \ge j \ge 0$, on notera
\[ \Glk[i]{j} = \Glk{i-1}\cdots\Glk{j+1}\Glk{j}. \]
Un \ndef{morphisme de \oo-précatégories} est un morphisme d'ensembles
globulaires qui respecte les applications $\comp_i^j$ et $\Glk{i}$.

Une \oo-précatégorie $X$ est une \ndef{\oo-catégorie stricte}
si elle satisfait en outre les axiomes suivants :
\begin{enumerate}[(C1)]
  \item Associativité \\
     pour tout triplet $(w, v, u)$ dans
    $(X_i, \Gls[j]{i}) \times_{X_j} (\Glt[j]{i}, X_i,
    \Gls[j]{i}) \times_{X_j} (\Glt[j]{i}, X_i)$ avec $i > j \ge 0$,
    on a
    \[ (w \comp_j^i v) \comp_j^i u = w \comp_j^i (v \comp_j^i u) \text{;} \]
  \item Loi d'échange\\
   pour tout quadruplet $(\delta, \gamma, \beta, \alpha)$ dans 
    \[ (X_i, \Gls[j]{i}) \times_{X_j} (\Glt[j]{i}, X_i, \Gls[k]{i})
    \times_{X_k} (\Glt[k]{i}, X_i, \Gls[j]{i}) \times_{X_j} (\Glt[j]{i},
    X_i),\]
    avec $i > j > k \ge 0$, on a
    \[ (\delta \comp^i_j \gamma) \comp^i_k (\beta \comp^i_j \alpha) = (\delta \comp^i_k \beta)
    \comp^i_j ( \gamma \comp^i_k \alpha) \text{;} \]
  \item Unités\\
   pour tout $u$ dans $X_i$ avec $i \ge 1$ et tout $j$ tel que $i > j \ge
   0$, on a
    \[ u \comp_j^i \Glk[i]{j}\Gls[j]{i}(u) = u = \Glk[i]{j}\Glt[j]{i}(u)
    \comp_j^i u \text{;} \]
  \item Fonctorialité des unités \\
   pour tout couple 
    $(v, u)$ dans $(X_i, \Gls[j]{i}) \times_{X_j} (\Glt[j]{i} ,X_i)$
    avec $i > j \ge 0$, on a
    \[ \Glk{i}(v \comp_j^i u) = \Glk{i}(v) \comp_j^{i+1} \Glk{i}(u). \]
\end{enumerate}
On notera $\wcat$ la sous-catégorie pleine de la catégorie des
\oo-précatégories dont les objets sont les \oo-catégories strictes.
\end{paragr}

\begin{paragr}
Soient $C$ une \oo-catégorie stricte, $u$ une $i$-flèche pour $i \ge 1$ et $j$ un
entier tel que $0 \le j < i$. On dit que $u$ \ndef{admet un
$\comp^i_j$-inverse} s'il existe une $i$-flèche $v$ telle que
\[
\Gls[j]{i}(v) = \Glt[j]{i}(u), \quad
\Glt[j]{i}(v) = \Gls[j]{i}(u), \quad
u \comp_j^i v = \Glk[i]{j}(\Glt[j]{i}(u))
\quad\text{et}\quad
v \comp_j^i u = \Glk[i]{j}(\Gls[j]{i}(u)).
\]
Le même argument qu'en théorie des groupes montre que si un tel $v$ existe,
il est unique. On appelle alors $v$ le \ndef{$\comp^i_j$-inverse} de $u$.

On dit qu'une \oo-catégorie stricte $C$ est un \ndef{\oo-groupoïde strict}
si pour tous entiers $i, j$ tels que $0 \le j < i$,  toute $i$-flèche de $C$
admet un $\comp^i_j$-inverse. On notera $\wgpd$ la sous-catégorie pleine de
$\wcat$ dont les objets sont les \oo-groupoïdes stricts.
\end{paragr}

\begin{prop}
Soit $C$ une \oo-catégorie stricte. Les assertions suivantes sont
équivalentes :
\begin{enumerate}
\item $C$ est un \oo-groupoïde strict ;
\item $C$ admet des $\comp^i_{i-1}$-inverses pour tout $i \ge 1$ ;
\item $C$ admet des $\comp^i_0$-inverses pour tout $i \ge 1$ ;
\item pour tout $i \ge 1$, il existe un entier $j$ tel que $0 \le j < i$
  pour lequel $C$ admet des $\comp^i_j$-inverses.
\end{enumerate}
\end{prop}

\begin{proof}
Par récurrence, il suffit de montrer que pour tous $i > j > k \ge 0$, si $C$
admet des $\comp^j_k$-inverses, alors $C$ admet des $\comp^i_k$-inverses si
et seulement si $C$ admet des $\comp^i_j$-inverses. En utilisant le fait
que le $2$-graphe
\[
\xymatrix{
C_i \ar@<.6ex>[r]^-{\Gls[j]{i}} \ar@<-.6ex>[r]_-{\Glt[j]{i}} &
C_j \ar@<.6ex>[r]^-{\Gls[k]{j}} \ar@<-.6ex>[r]_-{\Glt[k]{j}} &
C_k
}
\]
est naturellement muni d'une structure de $2$-catégorie, on se ramène au cas
$k = 0$, $j = 1$ et $i = 2$. Le résultat est alors une conséquence du lemme
suivant.
\end{proof}

\begin{lemme}
Soit $C$ une $2$-catégorie dont les $1$-flèches sont inversibles. Alors une
$2$\nobreakdash-flèches de $C$ est inversible pour la composition horizontale
(c'est-à-dire $\comp^2_0$) si et seulement si elle est inversible pour la
composition verticale (c'est-à-dire $\comp^2_1$).
\end{lemme}

\begin{proof}
Soit $\alpha : u \to v$ une $2$-flèche de $C$. Si $\alpha$ admet un inverse
$\alpha^\ast$ pour la composition horizontale, alors on vérifie immédiatement que
$v \comp_0 \alpha^\ast \comp_0 u$ est un inverse pour la composition verticale.
Réciproquement, si $\alpha$ admet un inverse~$\alpha^{-1}$ pour la composition
verticale, alors $v^{-1} \comp_0 \alpha^{-1} \comp_0 u^{-1}$ est un inverse pour
la composition horizontale.
\end{proof}

\section{Équivalences faibles de $\infty$-groupoïdes stricts}
\label{section:equiv_faible_wgpd}

\begin{paragr}\label{paragr:def_equiv_faible_wgpd}
Soient $G$ un \oo-groupoïde strict et $u, v$ deux $n$-flèches pour $n \ge
0$. Une homotopie de $u$ vers $v$ est une $(n+1)$-flèche $u \to v$. Si une
telle homotopie existe, on dira que $u$ est homotope à $v$ et on écrira
$u \tildeh[n] v$. Il est immédiat que si $u$ est homotope à $v$, alors $u$
et $v$ sont parallèles.
\end{paragr}

\begin{lemme}
Soit $G$ un \oo-groupoïde strict. Pour tout $n \ge 0$, la relation
$\tildeh[n]$ est une relation d'équivalence. De plus, si $n \ge 1$, cette
relation est compatible avec la composition $\comp^n_{n-1}$.
\end{lemme}

\begin{proof}
Soit $u$ une $n$-flèche de $G$. La $(n+1)$-flèche
$\Glk{n}(u)$ est une homotopie de $u$ vers $u$.
La relation $\tildeh[n]$ est donc réflexive.

Soit maintenant $v$ une deuxième $n$-flèche de $G$ et soit $h : u \to v$ une homotopie.
La $(n+1)$-flèche $\Glw{n+1}(h)$ est une homotopie de $v$ vers $u$. 
La relation $\tildeh[n]$ est donc symétrique.

Soit maintenant $w$ une troisième $n$-flèche de $G$ et $k : v \to w$ une
deuxième homotopie.  La $(n+1)$-flèche $k \comp^{n+1}_n h$ est une homotopie de
$u$ vers $w$.  La relation $\tildeh[n]$ est donc transitive.

Enfin, soit
\[
\UseAllTwocells
\xymatrix@C=3pc{
\rtwocell^u_{u'}{\,h}
&
\rtwocell^v_{v'}{\,k}
&
}
\]
un diagramme dans $G$, où les flèches simples sont des $n$-flèches avec $n \ge
1$ et les flèches doubles sont des $(n+1)$-flèches. La $(n+1)$-flèche
$k \comp^{n+1}_{n-1} h$ est une homotopie de
$v \comp^{n+1}_n u$ vers $v' \comp^{n+1}_n u'$. La relation $\tildeh[n]$
est donc compatible à la composition $\comp^n_{n-1}$.
\end{proof}

\begin{paragr}
Soit $G$ un \oo-groupoïde strict. Pour $n \ge 0$, on notera $\overline{G_n}$
le quotient de $G_n$ par la relation d'équivalence $\tildeh[n]$.

Fixons maintenant $n \ge 1$. Les applications
\[
 \Gls{n}, \Glt{n} : G_n \to G_{n-1},
 \quad
 \Glk{n-1} : G_{n-1} \to G_n,
\]
induisent des applications
\[
 \Gls{n}, \Glt{n} : \overline{G_n} \to G_{n-1},
 \quad
 \Glk{n-1} : G_{n-1} \to \overline{G_n}.
\]
De plus, par le lemme ci-dessus, l'application
\[ \comp^n_{n-1} : G_n \times_{G_{n-1}} G_n \to G_n \]
induit une application
\[
 \comp^n_{n-1} : \overline{G_n} \times_{G_{n-1}} \overline{G_n} \to
 \overline{G_n}.
\]
Il est immédiat que le graphe
\[
\xymatrix{
\overline{G_n} \ar@<.6ex>[r]^-{\Gls{n}} \ar@<-.6ex>[r]_-{\Glt{n}} &
G_{n-1},
}
\]
muni des applications
\[
\comp^n_{n-1} : \overline{G_n} \times_{G_{n-1}} \overline{G_n} \to
\overline{G_n}
\quad\text{et}\quad
\Glk{n-1} : G_{n-1} \to \overline{G_n}
\]
définit un groupoïde. Nous noterons $\PI_n(G)$ ce groupoïde.

Il est immédiat que $\PI_n$ définit un foncteur de la catégorie des
\oo-groupoïdes stricts vers la catégorie des groupoïdes.
\end{paragr}

\begin{paragr}
Soit $G$ un \oo-groupoïde strict. L'ensemble des \ndef{composantes connexes}
de $G$ est l'ensemble
\[ \pi_0(G) = \pi_0(\PI_1(G)) = \overline{G_0}. \]

Soient $n \ge 1$ et $u, v$ deux $(n - 1)$-flèches parallèles de $G$. On notera
\[
\pi_n(G, u, v) = \Hom_{\PI_n(G)}(u, v)
\quad\text{et}\quad
\pi_n(G, u) = \pi_n(G, u, u).
\]
La composition de $\PI_n(G)$ induit une structure de groupe sur $\pi_n(G, u)$.

Si $n \ge 1$ et $x$ est un objet de $G$, le \ndef{$n$-ième groupe d'homotopie}
de $(G, x)$ est le groupe
\[ \pi_n(G, x) = \pi_n(\Glk[n-1]{0}(x)). \]
L'argument de Eckmann-Hilton montre que ce groupe est abélien dès que $n \ge
2$.

On déduit de la fonctorialité de $\PI_n$ que :
\begin{itemize}
\item $\pi_0$ définit un foncteur de la catégorie des \oo-groupoïdes stricts
vers la catégorie des ensembles ;
\item pour $n \ge 1$, $\pi_n$ définit un foncteur de la catégorie des \oo-groupoïdes stricts munis
d'une $(n-1)$-flèche (ou d'un objet) vers la catégorie des groupes.
\end{itemize}
\end{paragr}

\begin{paragr}
Soit $f : G \to H$ un morphisme de \oo-groupoïdes stricts. On dira que
$f$ est une \ndef{équivalence faible de \oo-groupoïdes stricts} si
\begin{itemize}
\item l'application $\pi_0(f) : \pi_0(G) \to \pi_0(H)$ est une bijection ;
\item pour tout $n \ge 1$ et tout objet $x$ de $G$, le morphisme
\hbox{$\pi_n(f, x) : \pi_n(G, x) \to \pi_n(H, f(x))$} est un isomorphisme.
\end{itemize}
\end{paragr}

\begin{prop}\label{prop:eqf_defs}
Soit $f : G \to H$ un morphisme de \oo-groupoïdes stricts. Les
conditions suivantes sont équivalentes :
\begin{enumerate}
\item $f$ est une équivalence faible ;
\item l'application $\pi_0(f) : \pi_0(G) \to \pi_0(H)$ est une bijection,
et pour tout $n \ge 1$ et toute $(n-1)$-flèche $u$ de $G$,
le morphisme $f$ induit une bijection
\[ \pi_n(G, u) \to \pi_n(H, f(u))\text{ ;} \]
\item le foncteur $\PI_1(f) : \PI_1(G) \to \PI_1(H)$ est une équivalence de catégories, 
et pour tout $n \ge 2$ et tout couple $(u, v)$ de $(n-1)$-flèches parallèles
de $G$, le morphisme $f$ induit une bijection
\[ \pi_n(G, u, v) \to \pi_n(H, f(u), f(v))\text{ ;} \]
\item le foncteur $\PI_1(f) : \PI_1(G) \to \PI_1(H)$ est plein et essentiellement surjectif, et pour
tout $n \ge 2$ et tout couple $(u, v)$ de $(n-1)$-flèches parallèles de $G$,
le morphisme $f$ induit une surjection 
\[ \pi_n(G, u, v) \to \pi_n(H, f(u), f(v)). \]
\end{enumerate}
\end{prop}

\begin{proof}
L'implication $2 \Rightarrow 1$ est évidente. Montrons la réciproque.
Le cas $n = 1$ est évident.
Soient
$n \ge 2$ et $u$ une $(n-1)$-flèche de $G$. 
Posons $x = \Gls[0]{n-1}(u)$. On définit un isomorphisme
\[ \pi_{n}(G,x) \to \pi_n(G, u) \]
en envoyant une $n$-flèche $u : \Glk[n-1]{0}(x) \to \Glk[n-1]{0}(x)$
sur la $n$-flèche
$ \Glk{n-1}(u) \comp^n_0 u  : u \to u$.
Il est immédiat que le morphisme $f$ commute à cet isomorphisme, c'est-à-dire que le
carré 
\[
\xymatrix{
\pi_n(G, x) \ar[d] \ar[r] & \pi_n(G, u) \ar[d] \\
\pi_n(H, f(x))  \ar[r] & \pi_n(H, f(u)) \\
}
\]
est commutatif. L'application $\pi_n(G, u) \to \pi_n(H, f(u))$ est donc une
bijection pour $n \ge 2$. 

L'implication $3 \Rightarrow 2$ est évidente. Montrons la réciproque.
Soient $n \ge 1$ et $u, v$ deux $(n-1)$-flèches parallèles de $G$. Supposons qu'il
existe une $n$-flèche $\alpha : u \to v$ dans $G$. On définit alors un isomorphisme
\[ \pi_n(G, u) \to \pi_n(G, u, v) \]
en envoyant une $n$-flèche $\beta : u \to u$ sur la $n$-flèche
$\alpha \comp^n_{n-1} \beta  : u \to v$.
Il est immédiat que le morphisme $f$ commute à cet isomorphisme, c'est-à-dire que le
carré 
\[
\xymatrix{
\pi_n(G, u) \ar[d] \ar[r] & \pi_n(G, u, v) \ar[d] \\
\pi_n(H, f(u)))  \ar[r] & \pi_n(H, f(u), f(v)) \\
}
\]
est commutatif.
Ainsi, pour conclure, il suffit de montrer que s'il existe
une $n$-flèche $f(u) \to f(v)$ dans $H$, alors il existe une 
$n$-flèche $u \to v$ dans $G$. C'est clair pour \hbox{$n = 1$} par injectivité
de $\pi_0(f)$. Soit donc $n \ge 2$ et $\beta : f(u) \to f(v)$ une $n$-flèche de
$H$. Posons $x = \Gls{n-1}(u)$. La flèche $\Glk{n-1}(\Glw{n-1}(f(u)))
\comp^n_{n-2} \beta$ est une $n$-flèche de $H$ de source
$\Glw{n-1}(f(u)) \comp^{n-1}_{n-2} f(u) : f(x) \to f(x)$ et de but
$\Glw{n-1}(f(u)) \comp^{n-1}_{n-2} f(v) : f(x) \to f(x)$.
Puisque l'application
\[ \pi_{n-1}(G, x) \to \pi_{n-1}(H, f(x)) \]
est injective, les deux $(n-1)$-flèches $\Glw{n-1}(u) \comp^{n-1}_{n-2} u$ et
$\Glw{n-1}(u) \comp^{n-1}_{n-2} v$ sont égales dans $\pi_{n-1}(G,x)$. Puisque
$\PI_{n-1}(G)$ est un groupoïde, cela entraîne que $u = v$ dans
$\pi_{n-1}(G,x,y)$ et donc qu'il existe une $n$-flèche de $G$ de $u$ vers $v$.

L'implication $3 \Rightarrow 4$ est évidente. Montrons la réciproque.
Soient $n \ge 1$, $u,v$ deux $(n-1)$-flèches parallèles de $G$ et $\alpha, \beta$
deux $n$-flèches de $u$ vers $v$. Supposons qu'on ait $f(\alpha) = f(\beta)$
dans $\pi_n(H, f(u), f(v))$. Il existe alors une $(n+1)$-flèche de $H$ de
$f(\alpha)$ vers $f(\beta)$. Par surjectivité de l'application
\[ \pi_{n+1}(G, \alpha, \beta) \to \pi_{n+1}(H, f(\alpha), f(\beta)), \]
il existe une $(n+1)$-flèche de $G$ de $\alpha$ vers $\beta$. D'où $\alpha =
\beta$ dans $\pi_n(G, u, v)$.
\end{proof}

\begin{rem}
La dernière condition se reformule en la conjonction des propriétés
suivantes :
\begin{itemize}
\item pour tout objet $y$ de $H$, il existe un objet $x$ de $G$ tel que
$f(x)$ et $y$ soient homotopes ;
\item pour tout entier $n \ge 0$, tout couple $(u, v)$ de $n$-flèches parallèles de $G$ et toute
$(n+1)$-flèche $\beta : f(u) \to f(v)$ de $H$, il existe une 
$(n+1)$-flèche $\alpha : u \to v$ de $G$ telle que $f(\alpha)$ soit homotope à
$\beta$.
\end{itemize}
C'est exactement la notion d'équivalence faible de \oo-catégories strictes
exposée dans \cite{LMW}, restreinte aux \oo-groupoïdes stricts.
\end{rem}

\section{La catégorie homotopique}

\begin{paragr}
Un \ndef{localisateur} est la donnée d'une catégorie~$\C$ et d'une classe de
flèches $\W$ de~$\C$. On notera $(\C, \W)$ un tel localisateur et on appellera
$\W$ la classe des \ndef{équivalences faibles}.

La donnée d'un localisateur est suffisante pour définir les notions importantes
de la théorie de l'homotopie (comme par exemple, les notions de catégorie
homotopique, de limites homotopiques et de foncteurs dérivés).

Si $(\C, \W)$ est un localisateur, la \ndef{catégorie homotopique} de $(\C,
\W)$ est la localisation Gabriel-Zisman de $\C$ par $\W$, c'est-à-dire la
catégorie obtenue à partir de $\C$ en inversant formellement les morphismes de
$\W$. On notera cette catégorie $\Ho^{}_\W(\C)$ ou $\Ho(\C)$ si le contexte
rend claire la classe $\W$.

\forlang{A priori}, l'existence de la catégorie $\Ho(\C)$ nécessite un
changement d'univers (au sens de Grothendieck). Cependant, dans tous les
exemples que nous considèreront, la catégorie $\Ho(\C)$ existera sans
changement d'univers.
\end{paragr}

\begin{exs}
\
\let\sitem\item
\begin{myenumerate}
\item 
Soient $\C = \Top$ la catégorie des espaces topologiques et $\W_\infty$ la classe des
équivalences faibles topologiques, c'est-à-dire des applications continues $f :
X \to Y$ vérifiant les propriétés suivantes :
\begin{itemize}
\sitem l'application $\pi_0(f) : \pi_0(X) \to \pi_0(Y)$ est une bijection ;
\sitem pour tout $x$ dans $X$ et tout $i \ge 1$, le morphisme $\pi_i(f, x) : \pi_i(X, x)
\to \pi_i(Y, f(x))$ est un isomorphisme.
\end{itemize}

La \ndef{catégorie homotopique} $\Hot$ est la catégorie
$\Ho_{\W_\infty}(\Top)$. Un \ndef{type d'homotopie} est un objet de $\Hot$ à
isomorphisme canonique près.

\item
Fixons un entier $n \ge 0$. Soient $\C = \Top$ et $\W_n$ la classe des
$n$-équivalences faibles topologiques, c'est-à-dire des applications
continues $f : X \to Y$ vérifiant les propriétés suivantes :
\begin{itemize}
\sitem l'application $\pi_0(f) : \pi_0(X) \to \pi_0(Y)$ est une bijection ;
\sitem pour tout $x$ dans $X$ et tout $i$ tel que $1 \le i \le n$, le morphisme
de groupes $\pi_i(f, x) : \pi_i(X, x) \to \pi_i(Y, f(x))$ est un isomorphisme.
\end{itemize}

La \ndef{catégorie des $n$-types d'homotopie} $\Hot_n$ est la catégorie
$\Ho_{\W_n}(\Top)$. Un \ndef{$n$-type d'homotopie} est un objet de $\Hot$ à
isomorphisme canonique près.

\item Soit $n \ge 0$. Notons $\Top_{\le n}$ la sous-catégorie pleine de $\Top$
constituée des espaces topologiques possédant la propriété suivante : pour
tout $i > n$ et tout $x$ dans $X$, le groupe $\pi_i(X, x)$ est trivial. Cette
catégorie sera munie dans la suite des équivalences faibles topologiques
(qui coïncide évidemment avec les $n$-équivalences faibles)

\item 
La catégorie $\CW$ des CW-complexes sera dans la suite munie des équivalences
d'homotopie entre CW-complexes (qui sont également les équivalences faibles
topologiques entre CW-complexes par le théorème de Whitehead).

\item 
La catégorie $\pref{\Delta}$ des ensembles simpliciaux sera dans la suite munie
des équivalences faibles d'ensembles simpliciaux, c'est-à-dire des morphismes
d'ensembles simpliciaux qui s'envoient \forlang{via} le foncteur réalisation
géométrique $\vert\cdot\vert$ sur une équivalence d'homotopie entre
CW-complexes.

\item 
La catégorie $\Cat$ des petites catégories sera dans la suite munie des
équivalences faibles de catégories, c'est-à-dire des foncteurs s'envoyant
\forlang{via} le foncteur nerf $N$ sur une équivalence faible d'ensembles
simpliciaux.

\item 
La catégorie $\wgpd$ des \oo-groupoïdes stricts sera dans la suite munie des
équivalences faibles de \oo-groupoïdes stricts (qui ont été définies dans la
section précédente).

\item 
Soient $\A$ une catégorie abélienne et $k$ dans $\Z \cup \{-\infty\}$. On
notera $C_{\ge k}(\A)$ la catégorie des complexes homologiques d'objets de $\A$
en degré supérieur ou égal à $k$. Cette catégorie sera munie dans la suite
des quasi-isomorphismes, c'est-à-dire des morphismes de complexes induisant des
isomorphismes en homologie.

La catégorie $\Ho(C_{\ge k}(\A))$ n'est rien d'autre que la catégorie dérivée
$D_{\ge k}(\A)$ de $\A$.
\end{myenumerate}
\end{exs}

\begin{thm}
La chaîne de foncteurs
\[ \Cat \xrightarrow{N} \pref{\Delta} \xrightarrow{\vert\cdot\vert} \CW \to
\Top \]
induit une chaîne d'équivalences de catégories
\[ \Ho(\Cat) \xrightarrow{\overline{N}} \Ho(\pref{\Delta})
\xrightarrow{\overline{\vert\cdot\vert}} \Ho(\CW) \to \Ho(\Top) = \Hot. \]
\end{thm}

\begin{proof}
Voir le corollaire 3.3.1 de \cite{Illusie} pour $\overline{N}$. Pour les deux
autres foncteurs, voir \cite{Milnor} ou le chapitre VII de \cite{GZ}.
\end{proof}

\begin{rem}
Par le théorème précédent, on dispose donc de quatre définitions
équivalentes de $\Hot$. De même, un théorème similaire donne quatre définitions
équivalentes de $\Hot_n$. Le choix des définitions utilisant $\Cat$ est à la
base de la théorie des catégories test de Grothendieck
exposée dans \cite{GrothPS} et \cite{Maltsi}.
\end{rem}

\begin{paragr}
Les propriétés universelles de $\Hot$ et $\Hot_n$ entraînent l'existence
d'un triangle commutatif
\[
\xymatrix{
\Ho(\Top_{\le n}) \ar[rr]^-i \ar[dr]_j & & \Hot \ar[ld]^k \\
& \Hot_n & & .
}
\]
Des résultats classiques de la théorie de l'homotopie montrent que $i$ est
pleinement fidèle et que $j$ est une équivalence de catégories.
Ainsi, on peut identifier la catégorie $\Hot_n$ à la sous-catégorie pleine de
$\Hot$ dont les objets proviennent de $\Top_n$.

Si $X$ est un type d'homotopie, on appellera $k(X)$ le \ndef{$n$-type
d'homotopie associé à $X$}. On pourra le voir comme un objet de $\Hot$
\forlang{via} l'identification que l'on vient de décrire.
\end{paragr}

\section{La catégorie homotopique des $\infty$-groupoïdes stricts simplement connexes}

\begin{paragr}
On dira qu'un \oo-groupoïde strict $G$ est connexe s'il a au plus une
composante connexe. On dira que $G$ est simplement connexe s'il est connexe et
si pour tout objet $x$ de $G$, le groupe $\pi_1(G, x)$ est trivial. On notera
$\wgpdsc$ la sous-catégorie pleine de $\wgpd$ formée des \oo-groupoïdes
simplement connexes.

Soit $k$ un entier positif. On dira qu'un \oo-groupoïde strict $G$ est
\ndef{$k$-réduit} si pour tout $l$ entre $0$ et $k$, $G$ possède une unique
$l$-flèche. On notera $\wgpd_{\ge k+1}$ la sous-catégorie pleine de $\wgpd$
dont les objets sont les \oo-groupoïdes $k$-réduits.
\end{paragr}

\begin{prop}\label{prop:equiv_wgpdr_wgpd2}
Pour tout \oo-groupoïde strict simplement connexe $G$, il existe un
\oo-groupoïde strict $1$-réduit $G'$ et une équivalence faible de $G'$ vers
$G$.  En particulier, le foncteur $\Ho(\wgpdr) \to \Ho(\wgpdsc)$, induit par le
foncteur d'inclusion de $\wgpdr$ dans $\wgpdsc$, est essentiellement surjectif.
\end{prop}

\begin{proof}
Soit $G$ un \oo-groupoïde strict simplement connexe. Le choix d'un objet de $x$
de $G$ détermine un sous-\oo-groupoïde $G'$ de $G$ défini de la manière
suivante : $G'_0 = \{x\}$, $G'_1 = \{\Thk[1]{0}(x)\}$ et pour $n \ge 2$,
$G'_n$ est l'ensemble des $n$-flèches $f$ de $G$ telles que $\Gls[1]{n}(f) =
\Glk[1]{0}(x) = \Glt[1]{n}(f)$. L'inclusion $G' \to G$ est clairement une
équivalence faible.
\end{proof}

\begin{rem}\label{rem:equiv_sc_red}
Le foncteur $\Ho(\wgpdr) \to \Ho(\wgpdsc)$ est en fait une équivalence de
catégories. Par la proposition précédente, il suffit de montrer que ce foncteur
est pleinement fidèle. Cela résulte de l'existence de la structure de catégorie
de modèles de Brown-Golasi\'nski sur $\wgpd$ (voir \cite{BrownGolas} et
\cite{AraMetWGpd}) et du fait que les sous-catégories $\wgpdr$ et $\wgpdsc$
admettent des remplacements cofibrants et fibrants pour cette structure.
\end{rem}

\begin{paragr}
On notera $\Ab$ la catégorie des groupes abéliens et $\wgpd(\Ab)$ la catégorie
des \oo-groupoïdes stricts en groupes abéliens, c'est-à-dire des objets
\oo-groupoïdes stricts dans la catégorie des groupes abéliens. Par définition,
un objet $G$ de $\wgpd(\Ab)$ est un \oo-groupoïde strict dont, pour tout $n
\ge 0$, l'ensemble des flèches $G_n$ est muni d'une structure de groupes
abéliens, et dont tous les morphismes structuraux sont des morphismes de
groupes abéliens. On notera $\wgpdr(\Ab)$ la catégorie des \oo-groupoïdes
stricts en groupes abéliens $1$-réduits. Le foncteur d'oubli de la catégorie
des groupes abéliens vers la catégorie des ensembles induit un foncteur
$\wgpd(\Ab) \to \wgpd$ qui envoie un \oo-groupoïde strict en groupes abéliens
sur son \oo-groupoïde strict sous-jacent. Ce foncteur se restreint en un
foncteur $\wgpdr(\Ab) \to \wgpdr$. 
\end{paragr}

\begin{prop}\label{prop:wgpdr_hi}
Le foncteur $\wgpdr(\Ab) \to \wgpdr$ est un isomorphisme de catégories.
\end{prop}

\begin{proof}
Pour démontrer la proposition, il suffit de prouver que tout
\oo-grou\-poï\-de strict $1$-réduit est canoniquement un \oo-groupoïde
en groupes abéliens et que les morphismes de \oo-groupoïdes stricts
préservent cette structure abélienne.

Soit donc $G$ un tel \oo-groupoïde. Notons $\ast_i$ l'unique $e$-flèche de $G$
pour $e = 0, 1$. Soit $i \ge 2$. Si $u$ et $v$ sont deux $i$-flèches
de $G$, alors
\[ \Gls[e]{i}(u) = \Glt[e]{i}(u) = \Gls[e]{i}(v) = \Glt[e]{i}(v) = *_e \]
pour $e = 0, 1$. L'ensemble des $i$-flèches est donc muni de deux structures de
groupes, induites par les compositions $\comp^i_0$ et $\comp^i_1$. Par la loi
de l'échange, ces deux opérations sont compatibles. De plus, elles ont la même
unité puisque $\Glk[i]{0}(*_0) = \Glk[i]{1}(*_1)$. Par l'argument
d'Eckmann-Hilton, ces deux lois sont égales et commutatives. On notera $+$
cette opération. Le groupe $(G_i,+)$ est donc abélien.

Vérifions que les données de $G$ sont compatibles à cette structure
abélienne. Supposons toujours $i \ge 2$ et soient $u$, $v$
deux $i$-flèches. On a
\[
\Gls{i}(u + v) = \Gls{i}(u \comp_0^i v) = \Gls{i}(u) \comp_0^{i-1}
\Gls{i}(v) = \Gls{i}(u) + \Gls{i}(v).
\]
L'application $\Gls{i} : G_i \to G_{i-1}$ est donc un morphisme de groupes.
Il en est de même de l'application $\Glt{i}$. Soient $j$ tel que $0 \le j < i$ et $u',
v'$ deux autres $i$\nobreakdash-flèches vérifiant $\Gls[j]{i}(u) = \Glt[j]{i}(v)$ et
$\Gls[j]{i}(u') = \Glt[j]{i}(v')$.
On a
\[(u + u') \comp_j^i (v + v') = (u \comp_0^i u')
\comp_j^i (v \comp_0^i v') = (u \comp_j^i v) \comp_0^i (u' \comp_j^i v') = 
(u \comp_j^i v) + (u' \comp_j^i v').
\]
Ainsi $\comp_j^i$ est un morphisme de groupes. Supposons maintenant $i \ge 0$
et soient $u, v$ deux $i$-flèches. On a
\[
\Glk{i}(u + v) =
\Glk{i}(u \comp_0^i v) = \Glk{i}(u) \comp_0^{i+1} \Glk{i}(v) = \Glk{i}(u) +
\Glk{i}(v)
\]
et $\Glk{i}$ est un morphisme de groupes.

Montrons maintenant que cette structure abélienne est unique. Donnons-nous donc
une structure de \oo-groupoïde strict en groupes abéliens sur $G$. Fixons $i
\ge 2$ et notons $+'$ la loi de groupe sur $G_i$ de cette structure.  Puisque
l'application $\Glk[i]{0}$ respecte la loi $+'$, l'unité de~$+'$ est
$\Glk[i]{0}(\ast_0)$. En particulier, $+'$ et $\comp^i_0$ ont même unité. Par
ailleurs, puisque l'application $\comp^i_0 : G_i \times G_i \to G_i$ respecte
la loi $+'$, les lois $+'$ et $\comp^i_0$ sont compatibles.  Par l'argument
d'Eckmann-Hilton, elles sont donc égales. Ainsi, la structure abélienne est
uniquement déterminée par les compositions $\comp^i_0$.

De plus, puisque $+ = \comp^i_0$, un morphisme de \oo-groupoïdes stricts
$1$-réduits est automatiquement un morphisme de \oo-groupoïdes stricts en
groupes abéliens.
\end{proof}

\begin{paragr}\label{paragr:bourn}
Nous allons maintenant montrer que la catégorie des \oo-groupoïdes
stricts en groupes abéliens est canoniquement équivalente à la catégorie
$C_{\ge0}(\Ab)$ des complexes homologiques de groupes abéliens en degré
positif (ce résultat a été établi par Bourn dans \cite{Bourn}).

Soit $C$ un tel complexe. Notons $d_i : C_i \to C_{i-1}$ sa différentielle.
On associe à $C$ un \oo-groupoïde strict en groupes abéliens $G$ défini de
la manière suivante :
\begin{itemize}
\item le groupe $G_i$ des $i$-flèches est $C_i \oplus C_{i-1} \oplus \dots \oplus C_0$ ;
\item pour $i \ge 1$, le morphisme $\Gls{i} : G_i = C_i \oplus G_{i-1} \to G_{i - 1}$ est la
projection canonique ;
\item pour $i \ge 1$, le morphisme $\Glt{i} : G_i = C_i \oplus G_{i-1} \to G_{i - 1}$ est la
somme de $d_i$ et de la projection canonique ;
\item pour $i \ge 0$, le morphisme $\Glk{i} : G_i \to G_{i+1} = C_{i+1} \oplus G_i$ est l'injection
canonique ;
\item pour $i > j \ge 0$, le morphisme $\comp^i_j : (G_i, \Gls[j]{i})
\times_{G_j} (\Glt[j]{i}, G_i) \to
G_i$ est défini par
\[
(x_i, \dots, x_0) \comp^i_j (y_i, \dots, y_0) =
(x_i + y_i, \dots, x_{j+1} + y_{j+1}, y_j, \dots, y_0) \]
(notons que $(x_i, \dots, x_0, y_i, \dots, y_0)$ appartient à
$(G_i, \Gls[j]{i})
\times_{G_j} (\Glt[j]{i}, G_i)$
 si et seulement si on a
$(x_j, \dots, x_0) = (d_{j+1}(y_{j+1}) +
y_j, y_{j-1}, \dots, y_0)$).
\end{itemize}
On vérifie facilement que $G$ est bien un \oo-groupoïde strict en groupes
abéliens. Pour $i > j \ge 0$, le $\comp^i_j$-inverse d'un élément $(x_i,
\dots, x_0)$ de $G_i$ est donné par
\[ (-x_i, \dots, -x_{j+1}, d_{j+1}(x_{j+1}) + x_j, x_{j-1}, \dots, x_0). \]
De plus, si $f : C \to C'$ est un morphisme de complexes,
les composantes $f_k : C_k \to C'_k$ induisent pour tout $i$ positif, un
morphisme $G_i = C_i \oplus \dots \oplus C_0 \to G'_i = C'_i \oplus \dots
\oplus C'_0$. On vérifie facilement que ces morphismes définissent un
morphisme de \oo-groupoïdes stricts en groupes abéliens. On vient ainsi de
définir un foncteur $C_{\ge 0}(\Ab) \to \wgpd(\Ab)$.
\end{paragr}

\begin{prop}
Le foncteur $C_{\ge 0}(\Ab) \to \wgpd(\Ab)$ défini dans le paragraphe
précédent est une équivalence de catégories.
\end{prop}

\begin{proof}
Soit $G$ un \oo-groupoïde strict en groupes abéliens.
Posons $C_0 = G_0$ et $C_i = \Ker \Gls{i}$ pour $i \ge 1$.
Puisque le morphisme $\Gls{i} : G_i \to G_{i-1}$ admet $\Glk{i-1}$ comme
section, $G_i$ est canoniquement isomorphe à $C_i \oplus G_{i-1}$.
Ainsi, on a montré que l'ensemble globulaire sous-jacent à $G$ est isomorphe à
un \oo-graphe en groupes abéliens de la forme
\[
\xymatrix{
C_0 &
\ar@<.6ex>[l]^-{\Glt{1}}\ar@<-.6ex>[l]_-{\Gls{1}}  C_1 \oplus C_0 &
\ar@<.6ex>[l]^-{\Glt{2}}\ar@<-.6ex>[l]_-{\Gls{2}} \cdots &
\ar@<.6ex>[l]^-{\Glt{i}}\ar@<-.6ex>[l]_-{\Gls{i}} C_i \oplus C_{i-1} \oplus \cdots
\oplus C_0 & 
\ar@<.6ex>[l]^-{\Glt{i+1}}\ar@<-.6ex>[l]_-{\Gls{i+1}} \cdots
}.
\]
À travers cet isomorphisme, $\Gls{i}$ devient la projection canonique $C_i \oplus
G_{i-1} \to G_{i-1}$, et $\Glk{i-1}$ l'inclusion canonique $G_{i-1} \to C_i
\oplus G_{i-1}$. Puisque $\Glk{i-1}$ est également une section de
$\Glt{i}$, pour $(0, y)$ dans $C_i \oplus G_{i-1}$, on a
$\Glt{i}(0, y) = y$. Par ailleurs, l'identité $\Gls{i-1}\Glt{i} =
\Gls{i-1}\Gls{i}$ entraîne que, pour $(x, 0)$ dans $C_i \oplus G_{i-1}$,
on a $\Glt{i}(x, 0) = (d_i(x), 0)$, où $d_i$ est un morphisme de $C_i$ vers
$C_{i-1}$. Ainsi, $\Glt{i}$ est la somme de la projection canonique et de ce
morphisme $d_i$. L'identité $\Glt{i-1}\Glt{i} = \Glt{i-1}\Gls{i}$ implique
immédiatement qu'on a $d_{i-1}d_i = 0$. On a ainsi associé un complexe
$(C, d)$ à $G$.

On montre de même que si $g : G \to G'$ est un morphisme de \oo-groupoïdes
stricts en groupes abéliens, le morphisme $g_i : G_i \to G'_i$ se décompose
en une somme de deux morphismes $f_i \oplus g_{i-1} : C_i \oplus G_{i-1} \to
C'_i \oplus G'_{i-1}$ et que les morphismes $f_i$ définissent un morphisme
de complexes.

On a donc construit un foncteur $H : \wgpd(\Ab) \to C_{\ge 0}(\Ab)$. 
Montrons que celui-ci est un quasi-inverse du foncteur $K : C_{\ge 0}(\Ab)
\to \wgpd(\Ab)$ de l'énoncé. Il est évident que $HK$ est isomorphe à l'identité.
Montrons que $KH$ est isomorphe à l'identité. Soit donc $G$ un
\oo-groupoïde strict en groupes abéliens. On a déjà montré que $(KH)(G)$ et
$G$ ont des ensembles globulaires sous-jacents canoniquement isomorphes et
que cette isomorphisme est compatible aux unités. Pour conclure, il suffit
de montrer que ces données suffisent à déterminer les composition $\comp^i_j$.

Soient donc $C$ un complexe et $i > j \ge 0$ deux entiers. Soit $(x_i,
\dots, x_0, y_i, \dots, y_0)$ un élément de $G_i \times_{G_j} G_i$.
Rappelons que cela signifie que les relations
\[
x_j = d_{j+1}(y_{j+1}) + y_j\quad\text{et}\quad
x_i = y_i, \quad 0 \le i < j,
\]
sont satisfaites. On a alors
\[
\begin{split}
\MoveEqLeft (x_i, \dots, x_0) \comp^i_j (y_i, \dots, y_0) \\
& = 
(x_i, \dots, x_{j+1}, d_{j+1}(y_{j+1}) + y_{j}, y_{j-1}, \dots, y_0) \comp^i_j (y_i, \dots, y_0) \\
& = 
(x_i, \dots, x_{j+1}, 0, \dots 0) \comp^i_j (0,
\dots, 0)\,\, + \\
& \qquad\quad
(0, \dots, 0, d_{j+1}(y_{j+1}) + y_{j}, y_{j-1}, \dots, y_0) \comp^i_j (y_i, \dots, y_0) \\
& = 
(x_i, \dots, x_{j+1}, 0, \dots 0) +
(y_i, \dots, y_0) \\
& =
(x_i + y_i, \dots, x_{j+1} + y_{j+1}, y_j, \dots, y_0),
\end{split}
\]
où l'avant dernière égalité résulte des égalités
\[
\begin{split}
(0, \dots, 0) & = \Glk[i]{j}(\Gls[j]{i}(x_i, \dots, x_{j+1}, 0, \dots, 0)), \\
(0, \dots, 0, d_{j+1}(y_j) + y_j, y_{j-1}, \dots, y_0) & = 
\Glk[i]{j}(\Glt[j]{i}(y_i, \dots, y_0)),
\end{split}
\]
et de l'axiome des unités.
\end{proof}

\begin{prop}
Dans l'équivalence de catégories $C_{\ge0}(\Ab) \to \wgpd(\Ab)$ de la
proposition précédente, les équivalences faibles de \oo-groupoïdes et les
quasi-isomorphismes sont échangés.
\end{prop}

\begin{proof}
Soit $G$ un \oo-groupoïde strict en groupes abéliens et $n \ge 1$. Rappelons
qu'on note $\sim_n$ la relation d'équivalence d'homotopie des $n$-flèches.
Si $x$ est un objet de $G$, l'application $G_n \to G_n$, qui à $f$ associe
$f - \Glk[n-1]{0}(x)$, induit un isomorphisme de groupes de $\pi_n(G, x)$
vers $\pi_n(G, 0)$. Par ailleurs, on a
\begin{align*}
\pi_n(G, 0) 
& \simeq \{f \in G_n; \Gls{n}(f) = \Glt{n}(f) = 0 \}/\sim_n \\
& = \{f \in G_n; \Gls{n}(f) = \Glt{n}(f) = 0 \}/\{f \in G_n; f \sim_n
0\} \\
& = \{f \in G_n; \Gls{n}(f) = \Glt{n}(f) = 0 \}
/\{\Glt{n+1}(h); h \in G_{n+1}, \Gls{n+1}(h) = 0\} \\
& \simeq \{(f, 0) \in C_n \oplus G_{n-1}; d_n(f) = 0\}/\{d_{n+1}(h); (h, 0) \in
C_{n+1}\oplus G_n\} \\
& \simeq \{f \in C_n; d_n(f) = 0\}/\{d_{n+1}(h); h \in C_{n+1}\} \\
& = \Ker(d_n)/\Im(d_{n+1}) \\
& = H_n(C(G)).
\end{align*}
De même, on montre que $\pi_0(G) \simeq H_0(C(G))$.

La proposition résulte de la naturalité de ces isomorphismes.
\end{proof}

\begin{coro}
L'équivalence de catégories $C_{\ge0}(\Ab) \to \wgpd(\Ab)$ induit une
équivalence de catégories $D_{\ge0}(\Ab) \to \Ho(\wgpd(\Ab))$.
\end{coro}

\begin{rem}
Soit $k \ge 0$. L'équivalence de catégories $C_{\ge 0}(\Ab) \to \wgpd(\Ab)$
se restreint en une équivalence de catégories $C_{\ge k}(\Ab) \to \wgpd_{\ge
k}(\Ab)$ et celle-ci induit une équivalence de catégories $D_{\ge k}(\Ab)
\to \Ho(\wgpd_{\ge k}(\Ab))$.
\end{rem}

\begin{coro}\label{coro:equiv_wgpdr_C2Ab}
On a une équivalence de catégories canonique 
$\wgpdr \to C_{\ge 2}(\Ab)$. Celle-ci induit une équivalence de catégories
$\Ho(\wgpdr) \to D_{\ge 2}(\Ab)$.
\end{coro}

\begin{proof}
Cela résulte de la remarque ci-dessus et de la proposition
\ref{prop:wgpdr_hi}.
\end{proof}

\begin{rem}
D'après la remarque \ref{rem:equiv_sc_red} et le corollaire précédent, on a un
zigzag canonique d'équivalences de catégories
\[ \Ho(\wgpdsc) \leftarrow \Ho(\wgpdr) \rightarrow D_{\ge 2}(\Ab). \]
\end{rem}

\begin{paragr}
Soient $A$ un groupe abélien et $n$ un entier positif. On notera $\CK{A}{n}$
le complexe homologique de groupes abéliens concentré en degré $n$ de valeur
$A$. Si $n \ge 2$, on notera $\GK{A}{n}$ le \oo-groupoïde strict image de
$\CK{A}{n}$ dans l'équivalence de catégories du corollaire ci-dessus.
On peut décrire explicitement $\GK{A}{n}$ de la manière suivante :
\begin{itemize}
\item l'ensemble des $i$-flèches de $\GK{A}{n}$ est un singleton qu'on notera $0$
  pour $i < n$, et est $A$ pour $i \ge n$ ;
\item les applications $\Gls{i}$ et $\Glt{i}$ sont égales et valent
l'identité $\id{0}$ pour $i$ entre
$1$ et $n - 1$, la projection $A \to 0$ pour $i = n$ et l'identité $\id{A}$
pour $i \ge n + 1$ ;
\item l'application $\Glk{i}$ vaut $\id{0}$ pour $i$ entre $0$ et $n - 2$,
l'application $0 \to A$ correspondant à l'élément neutre de $A$ pour $i = n
- 1$, et $\id{A}$ pour $i \ge n$ ;
\item pour $i > j \ge 0$, l'application $\comp_j^i$ vaut $\id{0}$ pour
$i < n$, l'addition $A \times A \to A$ pour $i \ge n$ et $j < n$, et vaut
l'identité $\id{A}$ sinon (on a $\GK{A}{n}_i \times_{\GK{A}{n}_j} \GK{A}{n}_i =
A \times_A A \simeq A$ si $j \ge n$).
\end{itemize}

Il est immédiat que $\GK{A}{n}$ est simplement connexe et que si on note $\ast$
son unique objet, on a pour tout entier $k \ge 2$,
\[
\pi_k(\GK{A}{n}, \ast) =
\begin{cases}
A & \text{si $k = n$}, \\
0 & \text{sinon}.
\end{cases}
\]

On peut définir de manière analogue des \oo-groupoïdes stricts $\GK{E}{0}$ pour
$E$ un ensemble, et $\GK{G}{1}$ pour $G$ un groupe. Le \oo-groupoïde
$\GK{E}{0}$ a pour composantes connexes $E$, et pour tout objet $x$ de
$\GK{E}{0}$ et tout $k \ge 1$, le groupe $\pi_k(\GK{E}{0}, x)$ est trivial.
Quant à $\GK{G}{1}$, il possède un unique objet $\ast$ (il est en particulier
connexe), et vérifie pour tout $k \ge 1$,
\[
\pi_k(\GK{G}{1}, \ast) =
\begin{cases}
G & \text{si $k = 1$}, \\
0 & \text{sinon}.
\end{cases}
\]

On appellera les \oo-groupoïdes définis dans ce paragraphe les
\ndef{\oo-groupoïdes d'Eilenberg-Mac Lane}.
\end{paragr}

\begin{prop}
Soit $C$ un complexe dans $C_{\ge 2}(\Ab)$. Alors, dans $D_{\ge 2}(\Ab)$,
$C$ est isomorphe (non canoniquement) à $\prod_{n \ge 2} \CK{H_n(C)}{n}$.
\end{prop}

\begin{proof}
Nous allons utiliser la structure de catégorie triangulée de $D_{\ge
2}(\Ab)$ munie de l'auto-équivalence de catégories qui à un complexe $C$ associe le
complexe $C[1]$ défini par $C[1]_i = C_{i-1}$.

Soit $C$ un complexe dans $C_{\ge2}(\Ab)$. Notons $d_i : C_i \to C_{i-1}$ sa
différentielle. Pour tout entier $k$, on notera
\[
\setlength{\arraycolsep}{2pt}
\setcounter{MaxMatrixCols}{20}
\begin{matrix}
\tau^{}_{\ge k}(C) & = \quad \cdots & \to & C_{k+2} & \to & C_{k+1} 
& \to & \Ker d_k & \to & 0       & \to & 0        & \to & \cdots &\\
\widetilde{\tau}^{}_{\ge k}(C) & = \quad
\cdots & \to & C_{k+2} & \to & C_{k+1} 
& \to & X_k & \to & \Im d_k   & \to & 0       & \to & \cdots & \\
\tau^{}_{\le k}(C) & = \quad
\cdots  & \to & 0        & \to & 0 
& \to & \Coker d_{k+1}  & \to & C_{k-1}  & \to & C_{k-2}  & \to & \cdots & .
\end{matrix}
\]
Fixons $n \ge 2$. On a une suite exacte 
\[ 0 \to \widetilde{\tau}^{}_{\ge n+1}(C) \to C \to \tau^{}_{\le n}(C) \to 0 \]
dans $C_{\ge 2}(\Ab)$.
Par ailleurs, le morphisme canonique $\tau^{}_{\ge n+1}(C) \to \widetilde{\tau}^{}_{\ge
n+1}(C)$ est un quasi-isomorphisme. On dispose donc dans $D_{\ge 2}(\Ab)$ d'un
triangle distingué
\[ \tau^{}_{\ge n+1}(C) \to C \to \tau^{}_{\le n}(C) \to \tau^{}_{\ge
n+1}(C)[1] . \]
Or tout morphisme, d'un complexe concentré en degré inférieur ou égal à $n$,
vers un complexe concentré en degré supérieur ou égal à $n + 2$, est nul dans
$D_{\ge 2}(\Ab)$ car les $\Ext^{(n+2)-n}_\Ab = \Ext^2_\Ab$ sont triviaux.
Ainsi, le morphisme de connexion de $\tau^{}_{\le n}(C)$ vers $\tau^{}_{\ge n+1}(C)[1] =
\tau^{}_{\ge n+2}(C[1])$ est nul.  Par
conséquent, le triangle ci-dessus est scindé et on dispose donc d'une
section $\tau^{}_{\le n}(C) \to C$. En composant l'inclusion canonique
$H_n(C)[n] \to \tau^{}_{\le n}(C)$ et cette section, on obtient un
morphisme \[ H_n(C)[n] \to C \] de $D_{\ge 2}(\Ab)$ qui induit un
isomorphisme sur le $H_n$.

En passant à la somme directe, on obtient un morphisme
\[ \bigoplus_{n \ge 2} H_n(C)[n] \to C \]
de $D_{\ge 2}(\Ab)$.
Ce morphisme induit des isomorphismes sur tous les $H_n$ et est donc un
isomorphisme. Par ailleurs, on a $\bigoplus_{n \ge 2} H_n(C)[n] = \prod_{n
\ge 2} H_n(C)[n]$, d'où le résultat.
\end{proof}

\begin{rems}
\
\begin{myenumerate}
\item On a utilisé implicitement dans la preuve ci-dessus le fait que le
foncteur de localisation $C_{\ge 2}(\Ab) \to D_{\ge 2}(\Ab)$ commute aux
produits et aux sommes.
\item On démontre de manière similaire le résultat analogue dans $D_{\ge
k}(\Ab)$ pour $k$ quelconque.
\end{myenumerate}
\end{rems}

\begin{rem}
Un complexe $C$ dans $C_{\ge 2}(\Ab)$ est isomorphe à $\prod_{n \ge 2}
\CK{H_n(C)}{n}$ dans $C_{\ge 2}(\Ab)$ si et seulement si sa différentielle est
nulle. Dans l'équivalence de catégories entre $\wgpd_{\ge 2}$ et $C_{\ge
2}(\A)$, ces complexes correspondent aux \oo-groupoïdes stricts qui vérifient
$\Gls{n} = \Glt{n}$ pour $n \ge 1$. Ainsi, un \oo-groupoïde dans $\wgpdr$ est
isomorphe à un produit de \oo-groupoïdes d'Eilenberg-Mac Lane si et seulement
si ses sources et buts coïncident. Le résultat est également vrai dans $\wgpd$.
\end{rem}

\begin{thm}\label{thm:iso_sc_prod}
\
\begin{myenumerate}
\item 
Soit $G$ un \oo-groupoïde strict $1$-réduit d'unique objet $\ast$. On a
dans $\Ho(\wgpdr)$ un isomorphisme (non canonique)
\[ G \simeq \prod_{n \ge 2} \GK{\pi_n(G, \ast)}{n}. \]
\item
Soit $G$ un \oo-groupoïde strict simplement connexe. On a dans
$\Ho(\wgpd)$ un isomorphisme (non canonique)
\[ G \simeq \prod_{n \ge 2} \GK{\pi_n(G, x)}{n}, \]
où $x$ est un objet quelconque de $G$.
\end{myenumerate}
\end{thm}

\begin{proof}
Le premier point est une conséquence immédiate de la proposition précédente et
de l'équivalence des catégories $\Ho(\wgpdr)$ et $\Ho(C_{\ge 2}(\Ab))$
(corollaire \ref{coro:equiv_wgpdr_C2Ab}).

Le second point résulte du fait que tout \oo-groupoïde strict simplement
connexe est faiblement équivalent à un \oo-groupoïde strict $1$-réduit
(proposition \ref{prop:equiv_wgpdr_wgpd2}) et du fait que le foncteur
d'inclusion $\wgpdr \to \wgpd$ commute aux produits.
\end{proof}

\begin{rem}
\forlang{A priori}, pour rendre la preuve du second point ci-dessus
correcte, il faut considérer le produit de \oo-groupoïdes d'Eilenberg-Mac
Lane apparaissant dans l'énoncé comme un produit dans $\wgpd$ (par
opposition à dans $\Ho(\wgpd)$).  Cette distinction n'a en fait pas lieu
d'être car le foncteur de localisation $\wgpd \to \Ho(\wgpd)$ commute aux
produits. Cela résulte immédiatement du fait que tous les \oo-groupoïdes
stricts sont fibrants pour la structure de catégorie de modèles de
Brown-Golasi\'nski.
\end{rem}

\section{Types d'homotopie et $\infty$-groupoïdes stricts}

\begin{paragr}
On appellera \ndef{foncteur de réalisation de Simpson} (la notion est inspirée
de \cite{Simpson}) la donnée d'un foncteur $Q : \wgpd \to \Top$ muni d'une
application
\[ e : G_0 \to Q(G) \]
naturelle en $G$, induisant une bijection
\[ \pi_0(G) \to \pi_0(Q(G)), \]
et d'isomorphismes
\[
\pi_n(G, x) \to \pi_n(Q(G), e(x)),\quad n\ge1,
\]
naturels en $(G, x)$. 

Dans la suite, on se donne un foncteur de réalisation de Simpson $Q$. On notera
$R : \wgpd \to \Hot$ le composé $pQ$ où $p$ est le foncteur de localisation
$\Top \to \Hot$.
\end{paragr}

\begin{prop}
Le foncteur $Q$ commute aux produits de $\infty$-groupoïdes d'Eilenberg-Mac
Lane à équivalence faible près. Plus précisément, pour tout ensemble $A_0$,
tout groupe $A_1$, et tous groupes abéliens $A_n$ pour $n \ge 2$, le morphisme
canonique
\[ Q\Big(\prod_{n\ge0}\GK{A_n}{n}\Big) \to \prod_{n\ge0} Q(\GK{A_n}{n}) \]
est une équivalence faible topologique.
\end{prop}

\begin{proof}
Fixons un entier positif $n$. On notera
\[ p_n : \prod_{n\ge0}\GK{A_n}{n} \to \GK{A_n}{n}
\quad\text{et}\quad
q_n : \prod_{n\ge0}Q(\GK{A_n}{n}) \to Q(\GK{A_n}{n})
 \]
les projections canoniques.
On dispose du triangle commutatif suivant :
\[
\xymatrix@C=1pc@R=2pc{
Q\Big(\prod_{n\ge0}\GK{A_n}{n}\Big)  \ar[rr]^{f} \ar[rd]_{Q(p_n)} & &
\prod_{n\ge0}Q(\GK{A_n}{n}) \ar[dl]^{q_n} \\
& Q(\GK{A_n}{n}) & & \text{.}
}
\]

Pour $n = 0$, le morphisme $p_0$ induit une bijection sur le $\pi_0$.  Il en
est donc de même de $Q(p_0)$. Par ailleurs, $q_0$ induit également une
bijection sur le $\pi_0$. Il en est donc de même de $f$.

Soit $n \ge 1$. Le morphisme $p_n$ induit un isomorphisme sur le $\pi_n$
pour tout point de base. Le foncteur $Q(p_n)$ induit donc un isomorphisme
sur le $\pi_n$ pour tout point de base de la forme $e(x)$, où $x$ est un
objet de $\prod_{n\ge0}\GK{A_n}{n}$. Puisque $e$ induit une bijection sur le
$\pi_0$, l'application $Q(p_n)$ induit un isomorphisme sur le $\pi_n$ pour
tout point de base. Par ailleurs, l'application $q_n$ induit également un
isomorphisme sur le $\pi_n$ pour tout point de base. On en déduit que c'est
également le cas de $f$, d'où le résultat.
\end{proof}

\begin{rem}
Le foncteur de localisation $\Top \to \Hot$ commute aux produits. Cela
résulte, par exemple, du fait que tout espace topologique est fibrant pour la
structure de catégorie de modèles usuelle sur $\Top$. La proposition précédente
peut donc se reformuler en disant que le foncteur $R : \wgpd \to \Hot$ commute
aux produits de \oo-groupoïdes d'Eilenberg-Mac Lane.
\end{rem}

\begin{thm}\label{thm:im_ess_sc}
Les types d'homotopie simplement connexes dans l'image essentielle du
foncteur $R$ sont exactement les produits d'espaces d'Eilenberg-Mac Lane.
\end{thm}

\begin{proof}
Un type d'homotopie simplement connexe provient d'un \oo-groupoïde strict
simplement connexe. Par le théorème \ref{thm:iso_sc_prod}, un tel
\oo-groupoïde est relié par un zigzag d'équivalences faibles à un produit de
\oo-groupoïdes d'Eilenberg-Mac Lane. On en déduit immédiatement le résultat par
la proposition précédente.
\end{proof}

\begin{prop}\label{prop:pas_prod_em}
Soit $X$ un CW-complexe connexe de dimension finie $n \ge 1$. Si $X$ a un groupe
d'homotopie $\pi_m(X)$ non trivial pour $m > n$, alors le type d'homotopie
de $X$ n'est pas un produit d'espaces d'Eilenberg-Mac Lane. En particulier,
les sphères de dimension $n$ pour $n \ge 2$ n'ont pas le type d'homotopie
d'un produit d'espaces d'Eilenberg-Mac Lane.
\end{prop}

\begin{proof}
Si $X$ avait le type d'homotopie d'un tel produit, on aurait
\[
H_m(X) \simeq H_m\left(\prod_{1 \le k \le m+1} \K{\pi_k(X)}{k}\right).
\]
Écrivons
\[
X = \K{\pi_m(X)}{m} \times Y,\quad\text{où}\quad Y = \prod_{\substack{1 \le k
\le m+1\\k\neq m}} \K{\pi_k(X)}{k}.
\]
En appliquant la formule de Künneth à cette décomposition de $X$,
on obtient une injection de $H_m(\K{\pi_m(X)}{m}) \otimes H_0(Y) \cong
\pi_m(X)$ dans $H_m(X)$. Mais puisque $m > n$, $H_m(X)$ est trivial et il en
est donc de même de $\pi_m(X)$. Contradiction.

Ce résultat s'applique aux sphères de dimension $n \ge 2$ car
\[ 
\pi_3(\Sn{2}) = \Z
\quad\text{et}\quad
\pi_{n+1}(\Sn{n}) = \Z/2\Z,\quad n \ge 3.
\]
\end{proof}

\begin{coro}
Le foncteur $R$ n'est pas essentiellement surjectif.
\end{coro}

\begin{prop}\label{prop:img_ess_R}
L'image essentielle du foncteur $R$ est contenue dans la classe des espaces
dont chaque composante connexe a pour revêtement universel un
produit d'espaces d'Eilenberg-Mac Lane.
\end{prop}

\begin{proof}
\newcommand\cw{\textrm{cw}}
Soit $G$ un \oo-groupoïde strict. Quitte à décomposer $G$ en une somme sur
ses composantes connexes, on peut supposer que $G$ a exactement une
composante connexe. Choisissons un objet de $G$ et appelons $G'$ le
sous-\oo-groupoïde $1$-réduit de $G$ déterminé par cet objet. Notons $X =
Q(G)$ et $X' = Q(G')$. L'inclusion $i : G' \to G$ induit des isomorphismes
sur les $\pi_n$
pour $n \ge 2$ et il en est donc de même de $Q(i) : X \to X'$. On notera
$X_\cw$ (resp. $X'_\cw$) un remplacement cellulaire de $X$ (resp. de $X'$),
c'est-à-dire un CW-complexe muni d'une fibration triviale $r : X_\cw \to X$
(resp. $r' : X'_\cw \to X'$).  Soit $\pi : \widetilde{X} \to X_\cw$ le
revêtement universel de $X_\cw$.  La simple connexité de $X'$ et donc de
$X'_\cw$ entraîne l'existence d'une application continue $X'_\cw \to
\widetilde{X}$ rendant le triangle
\[
\xymatrix@R=1pc{
& & \widetilde{X} \ar[d]^\pi  \\
& & X_\cw \ar[d]^r  \\
X'_\cw \ar[r]_{r'} \ar[rruu] & X' \ar[r]_{Q(i)} & X
}
\]
commutatif. Puisque $X'_\cw$ et $\widetilde{X}$ sont simplement connexes,
et que $Q(i)r'$ et $r\pi$ induisent des isomorphismes sur les $\pi_n$ pour $n
\ge 2$, l'application $X'_\cw \to \widetilde{X}$ est une équivalence faible.
D'où le résultat par le théorème \ref{thm:im_ess_sc}.
\end{proof}

\section{Types d'homotopie et $3$-groupoïdes quasi-stricts}

Pour tenter de contourner le caractère non essentiellement surjectif d'un
tel foncteur de réalisation de Simpson, Kapranov et Voedvosky ont affaibli la
notion de $n$\nobreakdash-groupoïde en demandant seulement l'existence d'inverses
«~faibles~» (voir \cite{KapVoe}). Simpson a démontré dans \cite{Simpson}
que cela ne suffit pas. Dans ce paragraphe, nous exposons un autre
argument. Nous utiliserons la notion de $n$-groupoïde présentée par Street
dans \cite{Street}. Celle-ci est équivalente par le corollaire 4.4 de
\cite{KMV} à celle de Kapranov et Voedvosky. Nous nous plaçons ici dans le
cas $n = \infty$ qui est développé dans \cite{LMW}. 

Soient $C$ une \oo-catégorie stricte et $x, y$ deux $n$-flèches de $C$ pour
$n \ge 0$. On définit par coïnduction mutuelle les deux notions suivantes :
\begin{itemize}
\item les $n$-flèches $x$ et $y$ sont \ndef{faiblement homotopes} s'il
  existe une $(n+1)$-flèche \ndef{faiblement inversible} $u : x \to y$ ;
\item une $(n+1)$-flèche $u : x \to y$ est \ndef{faiblement inversible} s'il
existe une $(n+1)$-flèche $v : y \to x$ telle que $uv$ et $vu$ soient
\ndef{faiblement homotopes} à des identités.
\end{itemize}
On dira qu'une \oo-catégorie stricte est un \ndef{\oo-groupoïde
quasi-strict} si toutes ses $n$-flèches pour $n \ge 1$ sont faiblement
inversibles.

Soient $G$ une \oo-catégorie stricte et $n$ un entier positif. On montre
(proposition 6 de~\cite{LMW}) que la relation de faible homotopie sur $C_n$
est une relation d'équivalence. Si de plus $G$ est un \oo-groupoïde
quasi-strict, alors par définition, deux $n$-flèches $u, v$ de $G_n$ sont
faiblement homotopes si et seulement si elles sont homotopes. On en déduit
que la relation d'homotopie est une relation d'équivalence sur $G_n$. On
peut donc définir l'ensemble $\pi_0(G)$ et les groupes $\pi_n(G, x)$ pour
$x$ un objet de $G$, et donc la notion d'\ndef{équivalence faible de
\oo-groupoïdes quasi-stricts}, de la même manière que dans le cas strict.

Soit $G$ un $3$-groupoïde quasi-strict $1$-réduit. Par définition, $G$ est
une $3$-catégorie stricte $1$-réduite telle que les $2$-flèches aient un
$\comp^2_1$-inverse à une $3$-flèche près, et que les $3$-flèches aient un
$\comp^3_2$-inverse. En désuspendant $G$, on obtient une catégorie monoïdale
symétrique~$C$, avec une symétrie et des contraintes triviales, qui est un
groupoïde, et telle que le monoïde $\pi_0(C)$ est un groupe. La catégorie
$C$ est un champ de Picard (sur le point) au sens de l'exposé XVIII de
\cite{SGA4}. Ainsi, en vertu du lemme 1.4.13 de \opcit, il existe une
catégorie monoïdale symétrique $D$, avec une symétrie et des contraintes
triviales, qui est un groupoïde et telle que le monoïde $D_0$ est un
groupe, ainsi qu'une équivalence faible monoïdale symétrique de $D$ vers
$C$. En suspendant, on obtient un $3$-groupoïde strict $H$ et un
\emph{pseudo}-foncteur $H \to G$ qui est une équivalence faible.

Ainsi, tout $3$-groupoïde quasi-strict simplement connexe est faiblement
équivalent à un $3$-groupoïde strict \forlang{via} un
\emph{pseudo}-foncteur. Les $3$-groupoïdes quasi-stricts simplement
connexes ne modélisent donc pas plus de types d'homotopie que les $3$-groupoïdes
stricts simplement connexes. En particulier, ils ne modélisent pas le
$3$-type associé à la sphère de dimension $2$ (comme le montre la
démonstration de la proposition \ref{prop:pas_prod_em}).

Il a par contre été démontré indépendamment par Leroy (\cite{Leroy3Types}),
Joyal et Tierney, et Berger (\cite{BergerDLS}) que les $3$-groupoïdes
faibles (au sens de \cite{GPSTricat}) modélisent bien les types d'homotopie.

\section{Deux questions}

Ce texte ne répond pas à la question suggérée par son titre : quels sont les
types d'homotopie modélisés par les \oo-groupoïdes stricts ? Plus
précisément, nous proposons la question suivante :

\begin{question}
Soit $Q$ un foncteur de réalisation de Simpson. Quelle est l'image essentielle
du foncteur $R : \wgpd \to \Hot$ ?
\end{question}

La proposition \ref{prop:img_ess_R} donne une condition nécessaire pour être
dans l'image. Cette condition nécessaire n'est probablement pas suffisante et
il s'agirait de dégager une condition plus fine, probablement en termes
d'action du groupe fondamental sur les groupes d'homotopie supérieurs.

Une deuxième question naturelle est la généralisation de la question
précédente aux \oo-groupoïdes quasi-stricts. En particulier, les
\oo-groupoïdes quasi-stricts modélisent-ils plus de types d'homotopie que
les \oo-groupoïdes stricts ? Plus précisément, en notant $\wgpdqs$ la
sous-catégorie pleine de $\wcat$ formée des \oo-groupoïdes quasi-stricts,
nous proposons la question suivante :

\begin{question}
Le foncteur $\Ho(\wgpd) \to \Ho(\wgpdqs)$, induit par le foncteur d'inclusion
$\wgpd \to \wgpdqs$, est-il une équivalence de catégories ? 
\end{question}

\bibliographystyle{smfplain}
\bibliography{biblio}

\end{document}